\documentclass{scrartcl}
\usepackage{graphicx}
\usepackage{amssymb, amsmath, amsthm, amsxtra, xspace}

\newtheorem{theorem}{Theorem}
\newtheorem{lemma}{Lemma}
\newtheorem{proposition}{Proposition}
\newtheorem{corollary}{Corollary}
\newtheoremstyle{exStyle}
   {5pt}
   {3pt}
   {}
   {}
  {\bfseries}
   {.}
   {.5em}
   {}
\theoremstyle{exStyle}
\newtheorem*{example}{Example}

\usepackage[T1]{fontenc}
\usepackage[utf8]{inputenc}
\usepackage[USenglish]{babel}
\usepackage{lmodern}

\usepackage{natbib}

\renewcommand{\P}{\textbf{P}}

\newcommand{\R}{\mathbb{R}}
\newcommand{\Rd}{\mathbb{R}^{d}}
\newcommand{\Rone}{\mathbb{R}}

\newcommand{\g}{G(k,d)}

\newcommand{\E}{E}
\newcommand{\var}{\text{Var}}
\newcommand{\orig}{\mathbf{o}}
\renewcommand{\Pr}{\pi}
\newcommand{\ind}{\mathbf{1}}
\renewcommand{\d}{\textup{d}}
\newcommand{\e}{\textup{e}}
\begin{document}

\title{Anisotropic Poisson Processes of Cylinders
}


\author{Malte Spiess$^1$
         \and
        Evgeny Spodarev$^2$
}

\maketitle

\begin{abstract}
  Main characteristics of stationary anisotropic Poisson processes of cylinders
  (dilated $k$-dimensional flats) in $d$-dimensional Euclidean space are
  studied. Explicit formulae for the capacity functional, the covariance
  function, the contact distribution function, the volume fraction, and the
  intensity of the surface area measure are given which can be used directly in
  applications.%
\end{abstract}

\\[1ex]

\footnotetext[1]{
Institute of Stochastics, Ulm University, Helmholtzstr.  18, D--89069 Ulm, Germany \\
              Tel.: +49-731-5023528\\
              Fax: +49-731-5023649\\
              malte.spiess\@uni-ulm.de
}
\footnotetext[2]{
Institute of Stochastics, Ulm University, Helmholtzstr.  18, D--89069 Ulm, Germany \\
              Tel.: +49-731-5023530\\
              Fax: +49-731-5023649\\
              evgeny.spodarev\@uni-ulm.de
}
\section{Introduction}
\label{intro}

Porous fiber materials find vast applications in modern material technologies.
Their use ranges from light polymer-based non-woven materials, see
\cite{hbso03}, to fiber-reinforced textile and fuel cells as in \cite{MukWan06}.
Their porosity, percolation, acoustic absorption and liquid permeability are of
special interest. It is known that these properties depend to a great
extent on the microscopic structure of fibers, in particular, on the
orientation of a typical fiber. If all directions of fibers are equiprobable one
speaks of \textit{isotropy}. Many materials are made by pressing an isotropic
collection of fibers together thus producing strongly \textit{anisotropic}
structures. As examples, pressed non-woven materials used as an acoustic trim in
car production, see \cite{sprwo06}, paper making process as in \cite{CorKal61},
and \cite{mo:st93}, and gas diffusion layers of fuel cells, here
\cite{maroflle}, \cite{mhglkhhb} can be mentioned; see Figure~\ref{figToray}.
\begin{figure}
  \begin{center} {\includegraphics[scale=0.8]{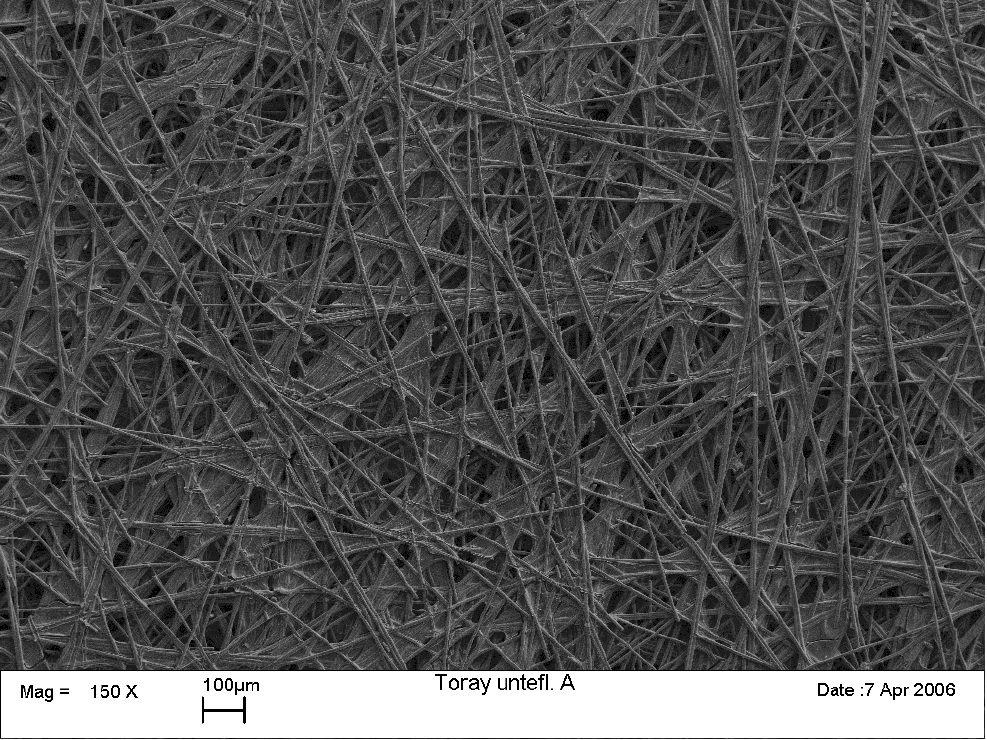}}
  \end{center}
  \caption{Microscopic structure of the gas diffusion layer as used in polymer
    electrolyte membrane fuel cells (courtesy of the Centre for Solar Energy and
    Hydrogen Research, Ulm)}\label{figToray}
\end{figure}

To quantify this dependence between the physical and the geometric structural
properties of porous materials, their \textit{intrinsic volumes} (sometimes also
called Minkowski functionals or quermassintegrals) are used. More formally,
porous fiber materials are usually modeled as homogeneous random closed sets
described in \cite{ma} and \cite{se82}. The mean volume and surface area of such
sets in an observation window averaged by the volume of the window are examples
of intensities of intrinsic volumes which are treated in detail in this paper.

The intention of this paper is to give formulae for cylinder processes which can
be used directly in applications, which is also demonstrated in the optimization
example. Thus the focus is on stationary Poisson processes which are the most
common in applications. A rather theoretical analysis can be found in the recent
paper by \cite{hoffm07}, where formulae for the curvature measures of a more
general non-stationary model of Poisson cylinder processes can be
found. In \cite{w87} the model for cylinders as used in this paper is
introduced, and curvature measures for different kinds of (not necessarily
Poisson) point processes are calculated. As opposed to that, in this paper
formulae for the covariance function, the contact distribution function, and a
different approach for the calculation of the specific surface area of Poisson
cylinder processes are worked out which have straightforward applied value.

As a model for fiber materials shown in Figure~\ref{figToray}, we consider
anisotropic stationary cylinder processes as homogeneous Poisson
point processes in the space of cylinders. Isotropic models of this kind (named also processes of
\textit{``thick'' fibers}, \textit{lamellae}, \textit{membranes} or
\textit{Poisson slices}) have been studied in detail, cf.\ \cite{ma},
\cite{se82}, \cite{davy78dis}, \cite{OhserMueckl00}. See \cite{schn87} for
further references. In the present paper, we generalize some of their results to
the anisotropic case.

After giving some preliminaries on cylinder processes (Section~\ref{sectFlats}),
we obtain formulae for the capacity functional, covariance function and contact
distribution function in Section~\ref{sectCapacity}. In
Section~\ref{sec:spec-surf-area}, we prove the formulae for the intensity of the
surface area measure of anisotropic stationary Poisson processes of
cylinders. Formulae for the intensities of other intrinsic volumes can
be found in the recent paper by \cite{hoffm07}. In the last section, we show how
the volume fraction of an anisotropic Poisson process of cylinders can be
maximized under certain constraints. In the solution, we use the formulae
obtained in previous sections.

Since the formulae obtained in
Sections~\ref{sectCapacity}--\ref{sec:spec-surf-area} are rather complex,
examples in the most interesting dimensions 2 and 3 are given, which can be
directly used in applications.

\section{Cylinder Processes} \label{sectFlats}

Let $\g$ be the \textit{Grassmann manifold} of all non-oriented
$k$-dimensional linear subspaces of $\Rd$, and $\mathfrak{G}$ the
$\sigma$-algebra of Borel subsets of $\g$ in its usual topology. Let
$\mathcal{C}$ $(\mathcal{K})$ be the set of all compact (compact convex) non-empty sets in
$\Rone^d$. Denote by $\mathcal{R}$ the \textit{convex ring}, i.e., the family of
all finite unions of non-empty compact convex sets. We provide these sets with the
Hausdorff metric, and denote the resulting Borel $\sigma$-algebra of
$\mathcal{R}$ by $\mathfrak{R}$.

Denote by $\nu_k(\cdot)$ the $k$-dimensional Lebesgue measure in $\Rone^k$, and
by $\mathcal{H}^k(\cdot)$ the $k$-dimensional Hausdorff measure. For any set $S
\subset \R^d$ denote by $S^\perp$ the linear subspace of the vectors which are
orthogonal to all elements of $S$. For $\xi$ being a $k$-dimensional flat (i.e.\
a $k$-dimensional linear subspace) we denote the $(d-k)$-dimensional Lebesgue
measure in $\xi^\perp$ by $\nu_{d-k}^\xi$. Let $\kappa_k$ ($\omega_k$) be the
volume (surface area) of a unit $k$-dimensional ball, respectively.

For a convex set $K \in \mathcal{K}$ and $\mathbf{x} \in \R^d$ let $p(K,
\mathbf{x})$ be the unique point in $K$ which is the closest to $\mathbf{x}$.
Then there exist measures $\Phi_k(K, \cdot)$ on $\mathcal{B}(\R^d)$, for $k =0,
\dots, d$ with
\begin{align*}
  \nu_d(\{\mathbf{x} \in K \oplus B_r(\orig): p(K, \mathbf{x}) \in B\})  
  = \sum_{k=0}^d r^{d-k} \kappa_{d-k} \Phi_k(K, B),
\end{align*}
where $K_1 \oplus K_2 = \{k_1 + k_2 | k_1 \in K_1, k_2 \in K_2\}$, and
  $B_r(o)$ is the ball of radius $r$ centered in the origin $o$.
  Furthermore we define $\Phi_k(\emptyset, B) = 0$ for all $B \in
  \mathcal{B}(\R^d)$. These measures are called \textit{curvature measures}.
  Since they are locally determined, they can be extended to functions with
  locally polyconvex sets as first argument in such a way that they remain
  additive. One should remark that these generalized curvatures measures are not
  necessarily positive, but signed measures. For a detailed
introduction, see \cite{SchneiWeil08}. The \textit{intrinsic volumes} of $K$ can
be defined as total curvature measures $V_k^d(K) = \Phi_k(K, \R^d)$ for $k =
0,\ldots,d$.

Following the approach introduced in \cite{w87}, we define a \textit{cylinder}
as the Minkowski sum of a flat $\xi \in G(k,d)$ and a set $K \subset \xi^\perp$
with $K \in \mathcal{R}$. Note that $K$ is not limited to sets with an
associated point in the origin. The flat $\xi$ is also called the
\textit{direction space} of $\xi \oplus K$ and $K$ is called the \textit{cross
  section} or \textit{base}. For a cylinder $Z = K \oplus \xi$ we define the
functions $L(Z) = \xi$ and $K(Z) = K$. Furthermore, define $\mathcal Z_k$ as the
set of all cylinders which have a $k$-dimensional direction space and base in
$\mathcal R$. For the volume of the cross-section of the cylinder we introduce
the notation $A(Z) = \nu_{d-k}^{L(Z)}(K(Z))$. By $S(K)$ we denote the surface
area of a set $K$. In the case of $K$ being the cross-section of a cylinder $K
\oplus L$ we shall use this notation for the surface area of $K$ in the space
$L^\perp$.

We call a measure $\varphi$ on $\mathcal Z_k$ \textit{locally finite} if
$\varphi(\{Z \in \mathcal Z_k | Z \cap K \ne \emptyset\}) < \infty$ for all $K
\in \mathcal C$. Let $\mathcal M(\mathcal Z_k)$ be the set of locally finite
counting measures on $\mathcal Z_k$ supplied with the usual $\sigma$-algebra $\mathfrak{M}$.
A point process $\Xi$ on $\mathcal Z_k$ which is a measurable mapping from a
probability space $(\Omega, \mathcal{F}, \P)$ into $(\mathcal M(\mathcal Z_k),
\mathfrak{M})$ is called a \textit{cylinder process}. Its distribution is given
by the probability measure $\P_\Xi \colon \mathfrak{M} \to [0,1]$, 
\mbox{$\P_\Xi(\cdot) = \P(\Xi \in \cdot)$.}
A cylinder process is called \textit{Poisson} if $\Xi(B)$ is Poisson distributed
with mean $\Lambda(B)$ for some locally finite measure $\Lambda$ on
$\mathcal{Z}_k$ and all Borel sets $B \subset \mathcal Z_k$, and $\Xi(B_1),
\Xi(B_2), \dots, \Xi(B_n)$ are independent for all disjoint Borel sets $B_1,
B_2, \dots, B_n \subset \mathcal Z_k$, and all $n \geq 2$, see details in
\cite{SchneiWeil08}. The measure $\Lambda$ is called the \textit{intensity
  measure} of $\Xi$. The Poisson cylinder process is called \textit{simple}, if
it has no multiple points. This is the case if and only if $\Lambda$ is diffuse.
For the rest of this paper we assume that $\Xi$ is a simple Poisson cylinder
process.
In this case, the union $U_\Xi = \cup_{Z \in \Xi} Z$ is a random closed set,
see \cite[p.~96]{SchneiWeil08}, where we denote by $Z \in \Xi$ the cylinders $Z$
in the support set of $\Xi$.

The cylinder process $\Xi$ is called \textit{stationary} if its distribution is
invariant with respect to translations in $\Rd$ and \textit{isotropic} if it is
invariant w.r.t.\ rotations about the origin. Let $\mathcal Z_k^o$ be the set of
all cylinders $K \oplus \xi$ with $\xi \in G(k, d), K \subset \xi^\perp$, and
for which the midpoint of the circumsphere of $K$ lies in the origin.

Following \cite{w87}, we define $i\colon (\mathbf{x}, Z) \mapsto \mathbf{x}+Z$
for $\mathbf{x} \in \Rd$ and $Z \in \mathcal Z_k^o$. If $\Xi$ is stationary,
then a number $\lambda \geq 0$ and a probability measure $\theta$ on
$\mathcal{Z}_k^o$ exist such that
\[
\Lambda(i(A \times C)) = \lambda \int_C \nu_{d-k}^{L(Z)}(A) \theta(\d Z)
\]
for all Borel sets $A \subset \Rd, C \subset \mathcal Z_k^o$. Then $\lambda$ is
called the \textit{intensity} and $\theta$ the \textit{shape distribution} of
$\Xi$.

As shown in \cite[p.~61]{kon:s}, $\theta$ can be decomposed further. Analogously
to $i$ define $j \colon (\varrho \times \xi) \mapsto \varrho \oplus \xi$ for
$\varrho \in \mathcal{R}$ and $\xi \in G(k,d)$. Then there exist a probability
measure $\alpha$ on $\mathfrak{G}$ (directional distribution of $\Xi$) and a
probability kernel $\beta \colon \mathfrak{R} \times G(k,d) \to [0,1]$ for which
$\beta(\cdot, \xi)$ is concentrated on subsets of $\xi^\perp$ such that for
arbitrary $R \in \mathfrak{R}$ and $G \in \mathfrak{G}$ the equation
\begin{align}\label{eq:decom.theta}
  \theta(j(R \times G)) = \int_G \beta(R, \xi) \alpha(\d \xi)
\end{align}
holds.

\section{Capacity functional and related characteristics} \label{sectCapacity}

In this section, we calculate the capacity functional (cf.\ \cite[p.~195]{skm})
for the union set $U_\Xi$ of the stationary Poisson process $\Xi$ of cylinders
with $k$-dimensional direction space introduced as above. As a corollary,
explicit formulae for the volume fraction, the covariance function, and the
contact distribution function of $U_\Xi$ follow easily. It is worth mentioning
that the resulting formula~(\ref{eq:T}) for the capacity functional generalizes
the formula in \cite[pp.~572-573]{se82}, given for Poisson slices in $\Rone^3$,
and a model with this capacity functional has already been proposed in
\cite[p.~148]{ma}.
\subsection{Capacity functional} \label{sectCapacity.subsectChoq}
For any random closed set $X$, the \textit{capacity functional}
$T_X(B) = \P(X \cap B \neq \emptyset), B \in \mathcal{C}$,
determines uniquely the distribution of $X$.

Let $\Pr_\eta(B)$ be the orthogonal projection of a set $B \subset \R^d$ along a
linear subspace $\eta \subset \Rd$ onto $\eta^\perp$.
\begin{lemma}
  \label{lem:T-Xi}
  The capacity functional of the union set $U_\Xi$ of the cylinder process $\Xi$
  is given by
  \begin{equation}
    \label{eq:T}
    T_{U_\Xi}(B) = 1 - \exp\left\{-\lambda \int_{\mathcal Z_k^o} \nu_{d-k}^{L(Z)}(-K(Z)
      \oplus {\Pr}_{L(Z)}(B) ) \, \theta(\d Z)\right\}.
  \end{equation}
\end{lemma}
\begin{proof}
  Let $B$ be a compact set in $\Rd$. Then by Fubini's theorem and
  \cite[p.~96]{SchneiWeil08}, we get
\begin{align*}
  1 - T_{U_\Xi}(B)
  &= \exp\{-\Lambda(\{Z \in \mathcal{Z}_k | Z \cap B \ne \emptyset\})\}
  = \exp\left\{-\int_{\mathcal{Z}_k} \mathbf{1}\{\tilde Z \cap B \ne \emptyset\} \Lambda(\d \tilde Z)\right\}\\
  &= \exp\left\{-\lambda \int_{\mathcal Z_k^o} \int_{L(Z)^\perp}
    \mathbf{1}\{(Z + \mathbf{x}) \cap B \ne \emptyset\} \, \d \mathbf{x} \, \theta(\d Z)\right\}\\
  &= \exp\left\{-\lambda \int_{\mathcal Z_k^o} \int_{L(Z)^\perp} 
    \mathbf{1}\{(K(Z) + \mathbf{x}) \cap {\Pr}_{L(Z)}(B) \ne \emptyset\} \, \d \mathbf{x} \, \theta(\d Z)\right\},
\end{align*}
where $\tilde Z = \mathbf{x} + Z$.

One can easily see that $K(Z)+\mathbf{x}$ hits
${\Pr}_{L(Z)}(B)$ if and only if $\mathbf{x}$ belongs to the Minkowski sum of $-K(Z)$ and
${\Pr}_{L(Z)}(B)$.

Thus we have
\begin{align*}
  1 - T_{U_\Xi}(B) &= \exp\left\{-\lambda \int_{\mathcal Z_k^o} \int_{L(Z)^\perp} \mathbf{1}\{(K(Z) + \mathbf{x}) \cap {\Pr}_{L(Z)}(B) \ne \emptyset\} \, \d \mathbf{x} \, \theta(\d Z)\right\}\\
  &= \exp\left\{-\lambda \int_{\mathcal Z_k^o} \int_{L(Z)^\perp} \mathbf{1}\{\mathbf{x} \in
    -K(Z) \oplus {\Pr}_{L(Z)}(B)\}\, \d \mathbf{x} \, \theta(\d Z)\right\}\\
  &= \exp\left\{-\lambda \int_{\mathcal Z_k^o} \nu_{d-k}^{L(Z)}(-K(Z) \oplus
    {\Pr}_{L(Z)}(B) ) \, \theta(\d Z)\right\}. 
\end{align*}
\end{proof}
A few remarks are in order.
\begin{itemize}
\item It follows from the local finiteness of $\Lambda$ that
  \begin{align} \label{eq:ineq-schwei}
  \int_{\mathcal Z_k^o} \nu_{d-k}^{L(Z)}(-K(Z) \oplus {\Pr}_{L(Z)}(B)
  ) \, \theta(\d Z) < \infty,
  \end{align}
  cf.~\cite[p.~96, Theorem 3.6.3.\ and remark]{SchneiWeil08}.
\item The choice of $k=0$ yields the capacity functional of the stationary
  Boolean Model $\Xi'$ with the primary grain $K$ and intensity $\lambda$,
  cf.\ \cite[p.~62]{ma}:
  \[
  T_{\Xi'}(B) = 1 - \e^{-\lambda \E \nu_d(-K \oplus B)}\,.
  \]
\item Another important special case is that of $K$ being a.s.\ a point. Then
  the model coincides with a $k$-flat process $\Xi''$, cf.\ \cite[p.~67]{ma} with
  the capacity functional
  \begin{align*}
    T_{\Xi''}(B) = 1 - \exp\left\{-\lambda \int_{\mathcal Z_k^o} 
      \nu_{d-k}^{L(Z)}({\Pr}_{L(Z)}(B)) \, \theta(\d Z)\right\}.
  \end{align*}


%
\item The case of $B = \{\orig\}$ yields the volume fraction $p = \P(\orig \in U_\Xi) =\linebreak
\E \nu_d(U_\Xi \cap [0, 1]^d)$ of $U_\Xi$:
\begin{equation} \label{eq:p}
  p = T_{U_\Xi}(\{\orig\}) = 1 - \exp\left\{- \lambda
    \int_{\mathcal Z_k^o} A(Z) \, \theta(\d Z)\right\}.
\end{equation}
A generalization of this formula can also be found in \textup{\cite{hoffm07}} in
the non-stationary setting.

Throughout this paper, we assume that $p > 0$, i.e.\ $\int_{\mathcal{Z}_k^o}
A(Z) \theta(\d Z) > 0$. Thus, we have $p \in (0, 1)$, cf.\
inequality~\eqref{eq:ineq-schwei}.
\end{itemize}

\subsection{Covariance function} \label{sectCapacity.subsectCov}
In the following we investigate the covariance function of $U_\Xi$. It is
defined as\linebreak $C_{U_\Xi}(\mathbf{h}) = \P(\orig, \mathbf{h} \in \Xi)$, $\mathbf{h} \in \Rd$,
cf.\ \cite[p.~68]{skm}.

Because of the relation $C_{U_\Xi}(\mathbf{h}) = \P(\orig, \mathbf{h} \in \Xi) = 2p - T_{U_\Xi}(\{\orig,
\mathbf{h}\})$ it is closely connected with the capacity functional of the set $B = \{\orig,
\mathbf{h}\}$, which is
\begin{equation}
  \label{eq:T0h}
  T_{U_\Xi}(\{\orig, \mathbf{h}\}) = 1 - \exp\left\{- \lambda \int_{\mathcal Z_k^o}
    \nu_{d-k}^{L(Z)}(\{\orig, {\Pr}_{L(Z)}(\mathbf{h})\} \oplus - K(Z)) \, \theta(\d Z)\right\}.
\end{equation}

Let $\gamma_A$ denote the \textit{covariogram} of a measurable set $A \subset L(Z)^\perp$
defined by
\[\gamma_A(\mathbf{x}) = \nu_{d-k}^{L(Z)}(A \cap (A - \mathbf{x}))
\]
for $\mathbf{x} \in L(Z)^\perp$.

\begin{lemma}
  For $\mathbf{h} \in \Rone^d$ we have
  \begin{equation}
    \label{eq:C(h)}
    \begin{split}
      C_{U_\Xi}(\mathbf{h}) =&\ 1 - 2 \exp\left\{-\lambda \int_{\mathcal Z_k^o} A(Z) \,
        \theta(\d Z)\right\}\\
      &\ + \exp\left\{-2\lambda \int_{\mathcal Z_k^o} A(Z) \,
        \theta(\d Z) + \lambda \int_{\mathcal Z_k^o} \gamma_{K(Z)}({\Pr}_{L(Z)}(\mathbf{h}))
        \, \theta(\d Z)\right\}.
    \end{split}
  \end{equation}
\end{lemma}
\begin{proof}
  Consider the term $\{\orig, {\Pr}_{L(Z)}(\mathbf{h})\} \oplus - K(Z) = - K(Z)
  \cup ({\Pr}_{L(Z)}(\mathbf{h}) - K(Z))$. Its volume is equal to
  \begin{align*}
    &\ \nu_{d-k}^{L(Z)}(- K(Z) \oplus \{\orig, {\Pr}_{L(Z)}(\mathbf{h})\})\\
    =&\ 2 A(Z) - \nu_{d-k}^{L(Z)}(K(Z) \cap
    (K(Z) -{\Pr}_{L(Z)}(\mathbf{h})))\\
    =&\ 2 A(Z) - \gamma_{K(Z)}({\Pr}_{L(Z)}(\mathbf{h})).
  \end{align*}

  Using equations~(\ref{eq:p}) and~(\ref{eq:T0h}), the covariance $C_{U_\Xi}(\mathbf{h})$
  rewrites
  \begin{align*}
    C_{U_\Xi}(\mathbf{h}) =&\ 2p - T_{U_\Xi}(\{\orig, \mathbf{h}\})\\
    =&\ 1 - 2 \exp\left\{- \lambda \int_{\mathcal Z_k^o} A(Z) \, \theta(\d Z)\right\}\\
    &+ \exp\left\{- \lambda \int_{\mathcal Z_k^o} \nu_{d - k}^{L(Z)}(- K(Z) \oplus
       \{\orig, {\Pr}_{L(Z)}(\mathbf{h})\})\, \theta(\d Z)\right\}\\
    =&\ 1 - 2 \exp\left\{-\lambda \int_{\mathcal Z_k^o} A(Z) \, \theta(\d Z)\right\}\\
    &\ + \exp\left\{-2\lambda \int_{\mathcal Z_k^o} A(Z)
      \, \theta(\d Z) + \lambda \int_{\mathcal Z_k^o} \gamma_{K(Z)}({\Pr}_{L(Z)}(\mathbf{h}))
      \, \theta(\d Z)\right\}. 
  \end{align*}
\end{proof}
\begin{example}
  In the following, we give an example of a cylinder process in two dimensions
  with cylinders of constant thickness $2a$ where the integrals
  in~\textup{(\ref{eq:C(h)})} can be calculated explicitly.

  Let $l$ be an arbitrary line through the origin, $\varphi$ the angle between
  the $x$-axis and $l^\perp$, and $\mathbf{h} = (r, \psi)$ a vector in polar
  coordinates. We use the notation $B_a(\orig) \times \varphi$ with $\varphi \in
  [0, \pi)$ for a cylinder with radius $a$ and direction space $l$. Since
  $|{\Pr}_{l}(\mathbf{h})| = r |\cos(\varphi - \psi)|$,
  formula~\textup{(\ref{eq:C(h)})} rewrites
\begin{align*}
  C_{U_\Xi}(\mathbf{h}) = 1 - 2 \e^{-2 \lambda a} + \e^{-4 \lambda a + \lambda I},
\end{align*}
where
\begin{align*}
  I &= \int_0^\pi (2a-|{\Pr}_{l}(\mathbf{h})|)\mathbf{1}\{|{\Pr}_{l}(\mathbf{h})| \le
    2a\}\, \theta(B_a(\orig) \times \d \varphi)\\
    &= \int_{\varphi \in [0, \pi]: |\cos(\varphi - \psi)| \le \frac{2a}{r}}
  (2a-r|\cos(\varphi-\psi)|) \, \theta(B_a(\orig) \times \d \varphi).
\end{align*}

In the isotropic case ($\theta(B_a(\orig) \times \d \varphi) = \d \varphi/\pi$) we can
choose $\psi$ arbitrarily, for example $\psi=\pi/2$. This yields
\begin{align*}
  I = \int_{\varphi \in [0, \pi]: \sin \varphi \le \frac{2a}{r}}
  (2a-r\sin \varphi) \frac{\d \varphi}{\pi}.
\end{align*}
In case $r \leq 2a$ this simplifies to
\begin{align*}
  I = 2a-r\int_0^\pi \sin \varphi \frac{\d \varphi}{\pi}
    = 2a-r\frac{2}{\pi}.
\end{align*}
And for $r > 2a$ we get
\begin{align*}
  I =&\ \frac{2a}{\pi} \left(\int_0^{\arcsin\frac{2a}{r}} \d \varphi +
      \int_{\pi-\arcsin\frac{2a}{r}}^\pi \d \varphi \right) +\\
    &+ \frac{r}{\pi} \left(\int_0^{\arcsin\frac{2a}{r}} (-\sin \varphi) \, \d \varphi
    + \int_{\pi-\arcsin\frac{2a}{r}}^\pi (-\sin \varphi) \, \d \varphi \right)\\
  =&\ \frac{4a}{\pi} \arcsin \left(\frac{2a}{r}\right) + \frac{2r}{\pi}
  \left(\cos \left(\arcsin \left(\frac{2a}{r}\right) \right)
    - 1 \right)\\
  =&\ 2a - \frac{4a}{\pi} \arccos \left(\frac{2a}{r}\right)
  - \frac{2r}{\pi}\left(1 - \sqrt{1- \left(\frac{2a}{r}\right)^2}\right),
\end{align*}
which gives us the final formula
\[
C_{U_\Xi}(\mathbf{h}) =
\begin{cases}
  1 - 2\e^{-2\lambda a} + \e^{-2 \lambda a - \frac{2 \lambda r}{\pi}},&\text{if } r \le 2a,\\
  1 - 2\e^{-2\lambda a} + \exp\bigg\{-2 \lambda a -\\
    \quad - \frac{\lambda}{\pi}\left(4a \arccos
      \left(\frac{2a}{r}\right) + 2r \left(1 - \sqrt{1 - \frac{4a^2}{r^2}}
      \right)\right)\bigg\},&\text{if } r > 2a.\\
\end{cases}
\]
\end{example}
The first derivative of $C_{U_\Xi}(\mathbf{h})$ will be needed later for the calculation of the
 intensity $\overline S_\Xi$ of the surface area measure of $U_\Xi$.
\begin{proposition}
  \label{pro:C_derivative}
  Suppose that $\Xi$ is a simple stationary Poisson cylinder process with shape
  distribution $\theta$ and $\int_{\mathcal Z_k^o} S(K(Z)) \theta(\d Z) < \infty$.
  Then the derivative of the covariance function in direction $\mathbf{h}$ at
  the origin is given by
  \begin{align*}  \label{eq:C_derivative}
    C_{U_\Xi}'(\orig, \mathbf{h})&\\
    = \lambda& \exp\left\{-\lambda \int_{\mathcal Z_k^o} A(Z) \, \theta(\d Z)\right\}
    \int_{\mathcal Z_k^o} \gamma_{K(Z)}'(\orig, {\Pr}_{L(Z)}(\mathbf{h}))
    \, [\mathbf{h}, L(Z)] \, \theta(\d Z)
    ,
  \end{align*}
  where $\gamma_A'(\orig, \eta)$ denotes the derivative of $\gamma_A$ at the origin
  in direction $\eta$, and $[\xi, \eta]$ is the volume of the parallelepiped
  spanned over the orthonormal bases of the linear subspaces of $\xi$ and
  $\eta$.
\end{proposition}
\begin{proof}
To simplify the notation, we shall also write $[\mathbf{x}, \eta]$
  for $[\xi, \eta]$ if $\xi$ is the line spanned by $\mathbf{x}$.
By (\ref{eq:p}), we have
\[
C_{U_\Xi}(\orig) = p = 1 - \exp\left\{-\lambda \int_{\mathcal Z_k^o} A(Z)
  \, \theta(\d Z)\right\}
\]
and thus
\[
C_{U_\Xi}(\mathbf{h}) - C_{U_\Xi}(\orig) = \exp\left\{-\lambda \int_{\mathcal Z_k^o} A(Z)
  \, \theta(\d Z)\right\}(\e^J-1),
\]
where
\begin{align*}
  J = \lambda \int_{\mathcal Z_k^o} [\gamma_{K(Z)}({\Pr}_{L(Z)}(\mathbf{h})) -
  A(Z)] \, \theta(\d Z).
\end{align*}
We observe that
$A(Z) - \nu_{d-k}^{L(Z)}(K(Z) \cap (K(Z) - {\Pr}_{L(Z)}(\mathbf{h})))$ is
equal to zero if\linebreak ${\Pr}_{L(Z)}(\mathbf{h}) = \orig$, and is less than or equal to
$|{\Pr}_{L(Z)}(\mathbf{h})| S(K(Z))$, otherwise.

This yields
\[
|J| \le \lambda |\mathbf{h}| \int_{\mathcal Z_k^o} S(K(Z)) \,
\theta(\d Z) = O(|\mathbf{h}|), \quad \mathbf{h} \rightarrow \orig.
\]
Thus we obtain $\e^J -1 = J + o(J) = J + o(|\mathbf{h}|)$ for $\mathbf{h} \rightarrow \orig$, and
\[
  C_{U_\Xi}'(\orig, \mathbf{h}) = \lim_{\mathbf{h} \rightarrow \orig} \frac{C_{U_\Xi}(\mathbf{h}) - C_{U_\Xi}(\orig)}{|\mathbf{h}|}
  = \exp\left\{-\lambda \int_{\mathcal Z_k^o} A(Z) \, \theta(\d Z)\right\}
    \left(\lim_{\mathbf{h} \rightarrow \orig} \frac J {|\mathbf{h}|}\right).
\]
So we need to investigate the behavior of $J /|\mathbf{h}|$ as $\mathbf{h} \rightarrow
\orig$. By the dominated convergence theorem, we get
\begin{align*}
  \lim_{\mathbf{h} \rightarrow \orig} \frac J {|\mathbf{h}|}
  &= \lambda \int_{\mathcal Z_k^o} \lim_{\mathbf{h} \rightarrow \orig}
  \Bigg( \frac{\nu_{d-k}^{L(Z)}(K(Z) \cap
    (K(Z) - {\Pr}_{L(Z)}(\mathbf{h})))
      - A(Z)}
    {\frac{|\mathbf{h}|}{|{\Pr}_{L(Z)}(\mathbf{h})|} |{\Pr}_{L(Z)}(\mathbf{h})|} \Bigg)
    \, \theta(\d Z)\\
  &= \lambda \int_{\mathcal Z_k^o} \lim_{t \rightarrow \orig}
  \Bigg( \frac{\nu_{d-k}^{L(Z)}(K(Z) \cap (K(Z) - t))
    - A(Z)}
    {|t|} \Bigg) |\cos \angle(\mathbf{h}, L(Z)^\perp)| \, \theta(\d Z)\\
  &= \lambda \int_{\mathcal Z_k^o} \gamma_{K(Z)}'(\orig,t) \, [\mathbf{h}, L(Z)]
  \, \theta(\d Z),
\end{align*}
where $|{\Pr}_{L(Z)}(\mathbf{h})|/|\mathbf{h}| = |\cos \angle(\mathbf{h}, L(Z)^\perp)|$, $t =
{\Pr}_{L(Z)}(\mathbf{h})$, and $\angle(\mathbf{h}, L(Z)^\perp)$ is the angle
between vector $\mathbf{h}$ and plane $L(Z)^\perp$.
\end{proof}

\subsection{Contact distribution function} \label{sectCapacity.subsectContact}
Let $B$ be an arbitrary compact set with $\orig \in B$ (called the structuring
element), and let $r > 0$. The contact distribution function (cf.\
\cite[p.~71]{skm}) $H_B(r) = \P(U_\Xi \cap r B \neq \emptyset\, |\, \orig \notin
U_\Xi)$ of the union set of the stationary Poisson cylinder process $\Xi$ with
structuring element $B$ and volume fraction $p \in (0, 1)$ can be calculated as
follows:
\begin{align}\label{eq:contactDistFunc}
  \begin{split}
  H_B(r) =&\ 1 - \frac{\P(U_\Xi \cap rB = \emptyset)}{1 - p}
  = 1 - \frac{1 - T_{U_\Xi}(rB)}{1 - p}\\
  =&\ 1 - \frac{\exp\left\{-\lambda \int_{\mathcal Z_k^o} \nu_{d-k}^{L(Z)}(-K(Z)
      \oplus {\Pr}_{L(Z)}(rB) ) \, \theta(\d Z)\right\}}
  {\exp\left\{- \lambda \int_{\mathcal Z_k^o} A(Z) \, \theta(\d Z)\right\}}\\
  =&\ 1 - \exp\left\{-\lambda \int_{\mathcal Z_k^o} \left[\nu_{d-k}^{L(Z)}(-K(Z)
    \oplus {\Pr}_{L(Z)}(rB)) - A(Z)\right] \, \theta(\d Z)\right\}.
  \end{split}
\end{align}
Further simplification of this formula is possible in some special cases.

Consider the contact distribution function $H_B$ with $B$ being a line segment
between the origin and a unit vector $\eta$. In this special case the contact
distribution function is called \textit{linear}. With a slight abuse of
notation we shall use a vector to represent the line segment between the origin
and the endpoint of the vector. It will be clear from the context whether the
vector or the line segment is meant.

\begin{lemma}
  \label{lem:linearCDF}
  If the probability kernel $\beta(\cdot,\xi)$
  (cf.~\textup{(\ref{eq:decom.theta})}) is concentrated on convex bodies and
  isotropic in the first argument for all $\xi \in G(k,d)$ then for a unit
  vector $\eta$ the linear contact distribution function of $U_\Xi$ is given by
  \begin{equation}
    \label{eq:lin-cont-distr}
    H_\eta(r) = 1 - \e^{-\lambda r C_o(\eta)}
  \end{equation}
  with
  \[
  C_o(\eta) = c_{d,k} \int_{G(k,d)} \int_{\mathcal{K} \cap \xi^\perp}
      S(K) \beta(\d K, \xi) [\xi, \eta] \alpha(\d \xi),
  \]
  $c_{d,k} = \frac{\omega_{d-k+1}}{2 \pi \omega_{d-k}}$, and $\mathcal{K} \cap \xi^\perp$ denotes the family of all convex bodies in $\xi^\perp$.
\end{lemma}
\begin{proof}
  \footnote{The idea of this proof goes back to an anonymous referee.}
  It follows from~(\ref{eq:contactDistFunc}) that~(\ref{eq:lin-cont-distr})
  holds iff
  \begin{align*}
    r \, C_o(\eta) = \int_{\mathcal Z_k^o} \left[\nu_{d-k}^{L(Z)}(-K(Z)
      \oplus {\Pr}_{L(Z)}(r \eta)) - A(Z)\right] \, \theta(\d Z).
  \end{align*}
  Using the notation introduced in \cite[p.~275-279]{schn} for mixed volumes
  (here all mixed volumes and surface measures are w.r.t.\ $L(Z)^\perp$) we
  calculate
  \begin{align*}
    \nu_{d-k}^{L(Z)}(-K(Z) \oplus {\Pr}_{L(Z)}(r \eta))& - A(Z)\\
    =&\ (d-k) V({\Pr}_{L(Z)}(r \eta), K(Z), \dots, K(Z))\\
    =&\ \frac{r}{2} \int_{S^{d-1} \cap L(Z)^\perp} |\langle u, \pi_{L(Z)}(\eta)\rangle| \, S_{d-k-1}(K(Z), \d u),
  \end{align*}
  where $\langle \cdot, \cdot\rangle$ denotes the scalar product, and
  $S_{d-k-1}(K(Z), \cdot )$ is the surface area measure of $K(Z)$ in $L(Z)^\perp$.

  Thus,
  \begin{align*}
    C_o(\eta) =&\ \frac{1}{2}\int_{\mathcal{Z}_k^o} \int_{S^{d-1} \cap L(Z)^\perp} 
      |\langle u, \pi_{L(Z)}(\eta)\rangle| \, S_{d-k-1}(K(Z), \d u) \,\theta(\d Z)\\
    =&\ \frac{1}{2}\int_{G(k,d)} \int_{\mathcal{K} \cap \xi^\perp} \int_{S^{d-1} \cap \xi^\perp} 
      |\langle u, \pi_\xi(\eta)\rangle| \, S_{d-k-1}(K, \d u) \,\beta(\d K, \xi)\,\alpha(\d \xi).
  \end{align*}

  Because of the rotation invariance of $\beta(\cdot, \xi)$, the value of the
  integral does not change if we replace $K$ with $\vartheta K$ for an arbitrary
  rotation $\vartheta$ in $\xi^\perp$. Furthermore, we get the following
  equation since the the surface area measure is invariant w.r.t.\
    rotations when they are applied to both arguments.
  \begin{align*}
    \int_{S^{d-1} \cap \xi^\perp} 
      |\langle u, \pi_\xi(\eta)\rangle| \, S_{d-k-1}(\vartheta K, \d u)
    =\int_{S^{d-1} \cap \xi^\perp} 
      |\langle \vartheta u, \pi_\xi(\eta)\rangle| \, S_{d-k-1}(K, \d u).
  \end{align*}
  Thus, integration over the group $\text{rot}(\xi^\perp)$ of rotations in
  $\xi^\perp$ equipped with the Haar probability measure leads to
  \begin{align*}
    &\int_{S^{d-1} \cap \xi^\perp}
      |\langle u, \pi_\xi(\eta)\rangle| \, S_{d-k-1}(K, \d u)\\
    =&\ \int_{\text{rot}(\xi^\perp)} \int_{S^{d-1} \cap \xi^\perp}
      |\langle u, \pi_\xi(\eta)\rangle| \, S_{d-k-1}(K, \d u) \, \d \vartheta\\
    =&\ \int_{S^{d-1} \cap \xi^\perp} \int_{\text{rot}(\xi^\perp)}
      |\langle \vartheta u, \pi_\xi(\eta)\rangle| \, \d \vartheta \, S_{d-k-1}(K, \d u)\\
    =&\ 2 c_{d,k} \, S(K) [\xi,\eta],
  \end{align*}
  where $c_{d,k}$ is the constant from the claim, and we used
  \cite[Corollary~5.2]{Spod02_1} for the last equality.
  
  This leads to
  \begin{align*}
    C_o(\eta) = c_{d,k} \int_{G(k,d)} \int_{\mathcal{K} \cap \xi^\perp}
      S(K) \beta(\d K, \xi) [\xi, \eta] \alpha(\d \xi).
  \end{align*}
\end{proof}
%
Now let the structuring element $B$ be the ball $B_1(\orig)$. In this case the contact
distribution function is called \textit{spherical}. It is obvious that
${\Pr}_{L(Z)}(B_r(\orig))$ is a ball of radius $r$ in the $(d-k)$-dimensional
subspace $L(Z)^\perp$. If $K(Z)$ is almost surely convex then the use of the
classical Steiner formula leads to
\begin{equation*}
  \E \nu_{d-k}^{L(Z)} (-K(Z) \oplus {\Pr}_{L(Z)}(B_r(\orig)))
  =\ \E A(Z) + \sum_{i=1}^{d-k} \kappa_i
    \E V_{d-k-i}^{d-k}(K(Z)) r^i,
\end{equation*}
which yields
\[
H_{B_1(\orig)}(r) = 1 - \exp\left\{-\lambda \sum_{i=1}^{d-k} \kappa_i r^i
  \int_{\mathcal Z_k^o} V_{d-k-i}^{d-k}(K(Z)) \, \theta(\d Z)\right\}.
\]
\begin{example}In what follows, the case of dimensions two and three is
  considered in detail. It is assumed that the conditions of
  Lemma~\textup{\ref{lem:linearCDF}} hold.

\begin{itemize}
\item For $d = 2, k = 1$ Lemma~\textup{\ref{lem:linearCDF}} yields
  \[
  C_o(\eta) 
  = c_{2,1} \int_{G(1,2)} \int_{\mathcal{K} \cap \xi^\perp}
      S(K) \beta(\d K, \xi) [\xi, \eta] \alpha(\d \xi)
  = \int_{G(1,2)} 2 [\xi, \eta] \alpha(\d \xi).
  \]
  Hence, it holds $H_{\eta}(r) = 1 - \exp\left\{-2 \lambda\, r \int_{G(1,2)} \,[\xi,
    \eta]\, \alpha(\d \xi)\right\}$, and so $H_{\eta}(r)$ does not depend on $K(Z)$.

  And for the structuring element being $B=B_1(\orig)$ one gets
  \[
  H_{B_1(\orig)}(r) = 1 - \exp\left\{-2 \lambda\, r \int_{\mathcal{Z}_1^o} V_0^1(K(Z)) \, \theta(\d Z)\right\}
  = 1 - \e^{-2 \lambda\, r}.
  \]
  Interestingly the result does not depend on the distribution of the cross
  section.

\item For $d = 3, k = 1$ we get
  \begin{align*}
    C_o(\eta) 
    =&\ \frac{2}{\pi} \int_{G(1,3)} \int_{\mathcal{K} \cap \xi^\perp}
      S(K) \beta(\d K, \xi) [\xi, \eta] \alpha(\d \xi)
  \end{align*}
  which yields
  \[
  H_{\eta}(r) 
  = 1 - \exp\left\{-\frac{2 \lambda r}{\pi} \int_{G(1,3)} \int_{\mathcal{K} \cap \xi^\perp}
      S(K) \beta(\d K, \xi) [\xi, \eta] \alpha(\d \xi)\right\}.
  \]
  For $K(Z) = B_a(\orig)$ we have
  \[
  C_o(\eta) = \frac{2 \pi a}{\pi} \int_{G(k,d)} [\xi, \eta]\, \alpha(\d \xi).
  \]
  Thus,
  \[
  H_{\eta}(r) = 1-\e^{-2 \lambda r a \int_{G(k,d)} [\xi, \eta]\, \alpha(\d \xi)}.
  \]

  And if the structuring element is the unit ball ($B=B_1(\orig)$) then
  \begin{align*}
    H_{B_1(\orig)}(r) =&\ 1 - \exp\left\{-\lambda \left(2 r \int_{\mathcal{Z}_1^o} V_1^2(K(Z))
        \, \theta(\d Z)
        + r^2 \int_{\mathcal{Z}_1^o} \kappa_2 \, \theta(\d Z)\right)\right\}\\
    =&\ 1 - \exp\left\{-\lambda \left(r \int_{\mathcal{Z}_1^o} S(K(Z)) \, \theta(\d Z)
        + r^2 \pi \right)\right\},
  \end{align*}
  where $S(K(Z))$ is the perimeter of $K(Z)$.

  If additionally $K(Z)$ is a ball of constant radius $a$ then
  \[
  H_{B_1(\orig)}(r) = 1 - \e^{-2 \pi a \lambda r - \pi \lambda r^2}.
  \]
\end{itemize}
\end{example}

\section{Specific surface area} \label{sec:spec-surf-area}

In the recent paper \cite{hoffm07}, the specific intrinsic volumes of a rather
general non-stationary cylinder process are given. In the stationary anisotropic
case, some of these formulae can be simplified. In this section, we give an
alternative proof for the specific surface area of the union set $U_\Xi$ of a
simple stationary anisotropic Poisson cylinder process $\Xi$ leading to a simpler
formula than that of \cite{hoffm07} which can be immediately used in
applications.

The specific surface area $\overline S_\Xi$ is defined as the mean surface area of
$U_\Xi$ per unit volume. More formally, consider the measure $S_{U_\Xi}(B) = \E
\mathcal{H}^{d-1}(\partial U_\Xi \cap B)$ for all Borel sets $B \subset
\Rone^d$. We assume that this measure is locally finite, i.e.\ $S_{U_\Xi}(B) <
\infty$ for all compact $B$. Sufficient conditions for this can be found in
Lemma~\ref{lem:S_Xi_finite}. Due to the stationarity of $\Xi$, the measure
$S_{U_\Xi}$ is translation invariant. By Haar's lemma, there exists a constant
$\overline S_\Xi \geq 0$ such that $S_{U_\Xi}(B) = \overline S_\Xi \,\nu_d(B)$ for all Borel
sets $B$, cf.\ \cite{amb90}. The factor $\overline S_\Xi$ is called the
\textit{specific surface area} of $U_\Xi$.

\begin{lemma}
  \label{lem:S_Xi_finite}
  The specific surface area $\overline S_\Xi$ of the union set $U_\Xi$ of a stationary
  anisotropic cylinder process $\Xi$ is finite if $\int_{\mathcal{Z}_k^o}
  S(K(Z)) \theta(\d Z) < \infty$.
\end{lemma}
\begin{proof}
  Let $B := B_1(\orig)$ be the unit ball about the origin. Then we calculate
  using the abbreviation $L_o = L(Z_o)$ and Campbell's theorem
  \begin{align*}
    &S_{U_\Xi}(B) = 
    \E \mathcal{H}^{d-1}(\partial U_\Xi \cap B)
      \le \E \sum_{Z \in \Xi} \mathcal{H}^{d-1}(\partial Z \cap B)
      = \int_{\mathcal{Z}_k} \mathcal{H}^{d-1}(\partial Z \cap B) \Lambda(\d Z)\\
      &= \lambda \int_{\mathcal{Z}_k^o} \int_{L_o^\perp} 
        \mathcal{H}^{d-1}((\partial Z_o + \mathbf{x}) \cap B)
        \, \nu_{d-k}^{L_o}(\d \mathbf{x}) \, \theta(\d  Z_o)\\
      &= \lambda \int_{\mathcal{Z}_k^o} \int_{L_o^\perp} \int_{\partial Z_o + \mathbf{x}}
        \ind_B(\mathbf{y}) \mathcal{H}^{d-1}(\d \mathbf{y})
        \, \nu_{d-k}^{L_o}(\d \mathbf{x}) \, \theta(\d  Z_o)\\
      &= \lambda \int_{\mathcal{Z}_k^o} \int_{L_o^\perp} \int_{\partial Z_o}
        \ind_B(\mathbf{y} + \mathbf{x}) \mathcal{H}^{d-1}(\d \mathbf{y})
        \, \nu_{d-k}^{L_o}(\d \mathbf{x}) \, \theta(\d  Z_o)\\
      &\le \lambda \int_{\mathcal{Z}_k^o} \int_{\partial Z_o} \int_{L_o^\perp}
        \ind_{\Pr_{L_o^\perp}(B)}(\Pr_{L_o^\perp}(\mathbf{y})) 
          \ind_{\Pr_{L_o}(B)}(\Pr_{L_o}(\mathbf{y}) + \mathbf{x})
          \nu_{d-k}^{L_o}(\d \mathbf{x})\, \mathcal{H}^{d-1}(\d \mathbf{y}) \, \theta(\d  Z_o)\\
      &= \lambda \int_{\mathcal{Z}_k^o} \int_{\partial Z_o}
        \ind_{\Pr_{L_o^\perp}(B)}(\Pr_{L_o^\perp}(\mathbf{y})) 
        \nu_{d-k}^{L_o}(\Pr_{L_o}(B))\, \mathcal{H}^{d-1}(\d \mathbf{y}) \, \theta(\d  Z_o)\\
      &= \lambda \nu_{d-k}^{L_o}(\Pr_{L_o}(B)) \int_{\mathcal{Z}_k^o}
        \mathcal{H}^{d-1}(\partial Z_o \cap (\Pr_{L_o^\perp}(B) \times L_o^\perp)) \, \theta(\d  Z_o)\\
      &= \lambda \kappa_{d-k} \int_{\mathcal{Z}_k^o}
        \nu_{d-k}^{L_o^\perp}(\Pr_{L_o^\perp}(B))
        \mathcal{H}^{d-k-1}(\partial K(Z_o)) \, \theta(\d  Z_o)\\
      &= \lambda \kappa_k \kappa_{d-k} \int_{\mathcal{Z}_k^o}
        S(K(Z_o)) \, \theta(\d  Z_o).
  \end{align*}
  This yields $\overline S_\Xi = S_{U_\Xi}(B) / \nu_d(B) < \infty $.
\end{proof}

The following results hold for any random closed set $X$ with realizations
almost surely from the \textit{extended convex ring} $\mathcal{S}$ which is
defined as the family of sets $B$ with $B \cap W \in \mathcal{R}$ for any convex
compact observation window $W$.

\begin{lemma}
  \label{lem:S_Xi_RCS}
  Let $X \in \mathcal{S}$ be an arbitrary stationary random closed set with
  finite specific surface area. Then the specific surface area of $X$ is given
  by
  \begin{equation}
    \label{eq:S_V}
    S_X = \frac{d \kappa_d}{\kappa_{d-1}} \int_{G(1,d)} \lambda(\xi) \d \xi,
  \end{equation}
  where $\d\xi$ is the Haar probability measure on $G(1, d)$, $\lambda(\xi) =
  \frac{1}{2} \E \Phi_0(X \cap \xi, B_1(\orig) \cap \xi)$ is the intensity of
  the number of connected components of $X \cap \xi$ on a line $\xi \in G(1,d)$.
\end{lemma}
\begin{proof}
  By Crofton's formula for polyconvex sets (cf.\ \cite[Th.~6.4.3]{SchneiWeil08}) and
  Fubini's theorem, we have
  \begin{align*}
    S_X =&\ \frac{1}{\kappa_d} \E \mathcal{H}^{d-1}(\partial X \cap B_1(\orig))
    = \frac{2}{\kappa_d} \E \Phi_{d-1}(X, B_1(\orig))\\
    =&\ \frac{2 \Gamma(\frac{d+1}2) \sqrt\pi}{\kappa_d \Gamma(d/2)} \E \int_{G(1, d)}
      \int_{\xi^\perp} \Phi_0(X \cap (\xi +\mathbf{x}), B_1(\orig) \cap (\xi +\mathbf{x}))
      \nu_{d-1}^{\xi}(\d \mathbf{x})\d \xi\\
    =&\ \frac{d \kappa_d}{\kappa_{d-1}} \int_{G(1, d)} \frac{1}{2} 
      \E \Phi_0(X \cap \xi, B_1(\orig) \cap \xi) \d \xi. 
\end{align*}
\end{proof}

The following result generalizes the well-known formula
\begin{align}
  \label{eq:S_X_isotrop}
  S_X=-\frac{d \kappa_d}{\kappa_{d-1}} C'_X(0)
\end{align}
\cite[p.~204]{skm}, for stationary, isotropic, and a.s.\ regular random closed
sets $X \in \mathcal{S}$ to the anisotropic case. A closed set is called
\textit{regular} if it coincides with the closure of its interior. Note that,
since in the isotropic case $C_X(h)$, $h \in \R^d$ depends only on the length of
$h$, and not on $h$ itself, in this formula $C_X$ is a function of a real
variable, namely the length of $h$. For the particular case of stationary
anisotropic random sets in $\R^3$ formula~\eqref{eq:S_Xi} can also be found
(without a rigorous proof) in \cite{berryman87}.
\begin{theorem}
  \label{the:S_Xi}
  Let $X$ be an a.s.\ regular  stationary random closed set with
  realizations from $\mathcal{S}$ and finite specific surface area. If $C_X(\mathbf{h})$
  is its covariance function then the specific surface area of $X$ is given by
  the formula
  \begin{equation}
    \label{eq:S_Xi}
    S_X = - \frac{d \kappa_d}{\kappa_{d-1}} \int_{G(1,d)} C_X'(\orig, \mathbf{r}_\xi) \d \xi,
  \end{equation}
  where $C_X'(\mathbf{h}, \mathbf{v})$ is the derivative of $C_X(\mathbf{h})$ at
  $\mathbf{h}$ in direction of unit vector $\mathbf{v}$, and $\mathbf{r}_\xi$ is
  a direction unit vector of a line $\xi \in G(1,d)$.
\end{theorem}
\begin{proof}
  For a stationary random closed set $U \subset \R$ from the extended convex
  ring denote by $-U$ the set reflected at the origin. Define a random variable
  $V$ which is uniformly distributed on $\{-1, 1\}$ and independent of $U$. The
  random closed set $U V$ is obviously isotropic, and thus
  formula~(\ref{eq:S_X_isotrop}) yields $S_{UV}=-2 C'_{UV}(0)$. Since $S_U =
  S_{UV}$ and $C_U'(0) = C_{UV}'(0)$, this means that $S_U=-2 C'_U(0)$.

  Applying this to $U = X \cap \xi$, $\xi \in G(1, d)$, we get $\lambda(\xi) =
  \frac{1}{2} S_{X \cap \xi} = -C_{X\cap\xi}'(0) = -C_X'(\orig, \mathbf{r}_\xi)$.
  Lemma~\ref{lem:S_Xi_RCS} completes the proof.
\end{proof}
If $X$ is an a.s.\ regular  two-dimensional stationary random closed set
with realizations in $\mathcal{S}$, formula~(\ref{eq:S_Xi}) simplifies to
\[
S_X = -\pi \int_0^\pi C_X'(\orig, \varphi) \frac{\d \varphi}{\pi} = -\int_0^\pi
C_X'(\orig, \varphi) \d \varphi.
\]
%
The following result is a direct corollary of
Proposition~\ref{pro:C_derivative}, Theorem~\ref{the:S_Xi}, and Fubini's
theorem.
\begin{corollary}
  Let $\Xi$ be a stationary Poisson cylinder process with intensity $\lambda$,
  shape distribution $\theta$ and cylinders with regular  cross-section
  $K(Z) \in \mathcal{R}$ for $\theta$-almost all $Z \in \mathcal{Z}_k$ and
  finite specific surface area. Then, the specific surface area of $U_\Xi$ is
  given by the formula
  \begin{align*}
    \overline S_\Xi
    =&\ - \lambda \frac{\kappa_d d}{\kappa_{d-1}} \int_{\mathcal Z_k^o}
    \int_{G(1,d)} \gamma'_{K(Z)}(\orig, {\Pr}_{L(Z)}(\mathbf{r}_\xi ))
      \ [\xi, L(Z)] \, \d \xi \, \theta(\d Z) \times\\
    &\ \times \exp\left\{-\lambda \int_{\mathcal Z_k^o} A(Z) \, \theta(\d Z)\right\}.
  \end{align*}
\end{corollary}
%

%
\begin{example}%
  Assume that $K(Z)$ is convex and regular for $\theta$-almost all $Z \in
  \mathcal{Z}_k$.
\begin{itemize}
\item For arbitrary $d$, and $k = d-1$ it holds for $\d \xi$-a.e.\ line $\xi \in
  G(1,d)$ that
  \[
  \gamma_{K(Z)}'(\orig, {\Pr}_{L(Z)}(\mathbf{r}_\xi)) = -1
  \]
  and
  \[
  \int_{G(1,d)} [\xi, L(Z)] \, \d \xi
  = \int_{G(d - 1,d)} [\xi^\perp, L(Z)] \, \d \xi
  = \frac{(d+1) \kappa_{d+1} \kappa_1}{d \kappa_d 2 \kappa_2}
  = \frac{(d+1) \kappa_{d+1}}{d \kappa_d \pi},
  \]
  see \textup{\cite[Corollary~5.2]{Spod02_1}}.

  This yields
  \begin{align*}
    \overline S_\Xi
    &= \lambda \frac{(d+1) \kappa_{d+1}}{\pi \kappa_{d-1}}
    \exp\left\{-\lambda \int_{\mathcal{Z}_{d-1}^o} \nu_1^{L(Z)}(K(Z)) \, \theta(\d Z)\right\}\\
    &= 2 \lambda \exp\left\{-\lambda \int_{\mathcal{Z}_{d-1}^o} \nu_1^{L(Z)}(K(Z)) \, \theta(\d Z)\right\}.
  \end{align*}
\item For $d=3$, $k = 1$, $K=B_a(\orig)$ it can be calculated that
  $\gamma_{K(Z)}'(\orig, {\Pr}_{L(Z)}(\xi)) = -\pi a$, $\int_{G(1, 3)}
  [\xi, L(Z)] \, \d \xi = 1/ 2$ (see also \textup{\cite[p.~298]{skm}},
  or \textup{\cite[Corollary~5.2]{Spod02_1}}), and thus we have
\[\overline S_\Xi = 4 \lambda
  \frac 1 2 \pi a \e^{-\lambda \pi a^2} = 2 \pi a \lambda \e^{-\lambda \pi a^2},
\]
  which coincides with the case of isotropic cylinders, compare
  \textup{\cite[p.~64]{OhserMueckl00}}.
\end{itemize}
\end{example}

\section{Optimization Example}
\label{sec:optimization-example}

In this section we show how the formulae from Sections~\ref{sectCapacity}
and~\ref{sec:spec-surf-area} can be applied to solve an optimization problem for
cylinder processes.

The following problem originates from the fuel cell research. The gas diffusion
layer of a polymer electrolyte membrane fuel cell is a porous material made of
polymer fibers (see Figure~\ref{figToray}) which can be modeled well by an
anisotropic Poisson process of cylinders in $\R^3$. In a gas diffusion layer,
the volume fraction of the polymer material lies between 70 and 80 percent, and
the directional distribution of fibers is concentrated on a small neighborhood
of a great circle of a unit sphere $\mathcal{S}^2$, i.e.\ all fibers are almost
horizontal. In order to optimize the water and gas transport properties, it is
desirable to have a relatively small variation of the size of pores in the
medium, where we define a pore at a point $x$ in the complement of $U_\Xi$ as
the maximal ball with center in $x$ which does not hit $U_\Xi$.

We investigate the following mathematical simplification of this problem, which
can be solved analytically in some particular cases.

For a fixed intensity $\lambda$ of the Poisson cylinder process $\Xi$, find a
shape distribution of cylinders $\theta$ which maximizes the volume fraction $p$
of $U_\Xi$ provided that the variance of the typical pore radius $H$ is
small. In other words, solve the optimization problem
\begin{align}\label{eq:optimization-problem}
  \begin{cases}
    p \rightarrow \max_\theta,\\
    \var H < \varepsilon,
  \end{cases}
\end{align}
where $H$ is a random variable with distribution function
$H_{B_1(\orig)}(r)$.

As it will be clear later, the condition on the directional distribution
$\alpha$ of fibers that all fibers are almost horizontal can be
neglected since the directional component of the shape distribution $\theta$ has
no influence on the solution.

To simplify the notation, let $c_s = \int_{\mathcal{Z}_1^o} S(K(Z)) \theta(\d Z)$
and $\Phi(x)$ be the distribution function of a standard normally distributed
random variable.

First we take a look at the moments of the pore radius $H$ (assuming that $r \ge
0$), remembering that $H_{B_1(\orig)}(r) = 1 - \e^{-\lambda (r c_s + r^2 \pi)}$ (as shown in an
example in Section~\ref{sectCapacity.subsectContact}) and thus the density of $H$ is equal to $\frac{\d }{\d  r}H_{B_1(\orig)}(r) = \lambda
(c_s + 2 \pi r) \e^{-\lambda (r c_s + r^2 \pi)}$. It holds
\begin{align*}
  \E H &= \int_0^\infty r \lambda (c_s + 2 \pi r) \exp \left(-\pi \lambda \left(r + \frac{c_s}{2 \pi}\right)^2 \right) \exp \left(\frac{c_s^2 \lambda}{4 \pi}\right) \,\d r\\
  &= \exp \left(\frac{c_s^2 \lambda}{4 \pi}\right) \lambda
    \int_\frac{c_s}{2 \pi}^\infty \left(r - \frac{c_s}{2 \pi}\right) \left(2 \pi r\right) \e^{-\pi \lambda r^2} \,\d r\\
  &= \exp \left(\frac{c_s^2 \lambda}{4 \pi}\right) \frac{1}{\sqrt{\lambda}}
 \left(1 - \Phi\left(c_s \sqrt{\tfrac{\lambda}{2 \pi}}\right) \right).
\end{align*}
Furthermore it can be calculated that
\begin{align*}
  \E H^2
  &= \exp \left(\frac{c_s^2 \lambda}{4 \pi}\right) \lambda \int_0^\infty r^2 (c_s + 2 \pi r) \exp \left(-\pi \lambda \left(r + \frac{c_s}{2 \pi}\right)^2 \right) \,\d r\\
  &= \frac{1}{\pi \lambda}
    - \exp \left(\frac{c_s^2 \lambda}{4 \pi}\right) \frac{c_s}{\pi \sqrt{\lambda}}\left(1 - \Phi\left(c_s \sqrt{\tfrac{\lambda}{2 \pi}}\right) \right).
\end{align*}
Defining $c_e = \exp \left(\frac{c_s^2 \lambda}{4 \pi}\right)$ and
$c_\Phi = \left(1 - \Phi\left(c_s \sqrt\frac{\lambda}{2 \pi}\right)\right)$,
this leads to
\begin{align*}
  \E H^2 - (E H)^2
  &= \frac{1}{\pi \lambda} - \frac{c_e c_\Phi c_s}{\pi \sqrt{\lambda}}
    - \frac{c_e^2 c_\Phi^2}{\lambda} \leq \varepsilon,
\end{align*}
multiplication with $\pi \lambda$ yields the equivalent condition
\begin{align*}
  1 - \sqrt{\lambda} c_e c_\Phi c_s - \pi c_e^2 c_\Phi^2 \leq \varepsilon \pi \lambda,
\end{align*}
which holds if and only if
\begin{align*}
  \left(c_e c_\Phi + \frac{\sqrt{\lambda} c_s}{2 \pi}\right)^2
    - \frac{\lambda c_s^2}{4 \pi^2}+ (\varepsilon \lambda - 1/\pi) \geq 0.
\end{align*}
This is always fulfilled if $\varepsilon \geq \frac{1}{\pi\lambda}$ and
$\frac{\lambda c_s^2}{4 \pi^2} - (\varepsilon \lambda - 1/\pi) \leq 0$ or, equivalently,
$c_s \leq 2 \pi \sqrt{\varepsilon - \frac{1}{\pi\lambda}}$.

In the following we always assume that $\varepsilon \geq \frac{1}{\pi\lambda}$
and replace the condition $\var H < \varepsilon$ by a stronger sufficient
condition
\begin{align}\label{eq:cond_S}
  c_s = \int_{\mathcal{Z}_1^o} S(K(Z)) \theta(\d Z) \leq 2\pi
    \sqrt{\varepsilon - \tfrac{1}{\pi\lambda}}.
\end{align}

Hence, \eqref{eq:optimization-problem} is reduced to the optimization problem
\begin{align}\label{eq:optimization-problem-new}
  \begin{cases}
    \int_{\mathcal{Z}_1^o} A(Z) \theta(\d Z) \rightarrow \max_\theta,\\
    \int_{\mathcal{Z}_1^o} S(K(Z)) \theta(\d Z) \le 2\pi
      \sqrt{\varepsilon - \tfrac{1}{\pi\lambda}}.
  \end{cases}
\end{align}

The solution of the optimization problem~\eqref{eq:optimization-problem-new}
yields cylinders with $\theta$-a.s.\ circular base. Notice that this solution
does not depend on the directional distribution component $\alpha$ of $\theta$.
Indeed, cylinders $Z$ can be replaced by cylinders $Z'$ which have the same
direction space and surface area ($S(K(Z)) = S(K(Z'))$) but are circular. Then the isoperimetric
inequality yields $A(Z') \ge A(Z)$. Thus, it holds that
\begin{align*}
  \int_{\mathcal{Z}_1^o} S(K(Z)) \theta(\d Z) = \int_{\mathcal{Z}_1^o} S(K(Z')) \theta(\d Z)
\end{align*}
and
\begin{align*}
  \int_{\mathcal{Z}_1^o} A(Z) \theta(\d Z) \le \int_{\mathcal{Z}_1^o} A(Z') \theta(\d Z),
\end{align*}
which means that the circular version is at least not worse than the original
version.

Thus, we assume that the cylinders are $\theta$-a.s.\ circular and denote the
radius of a cylinder $Z$ by $R(Z)$. It follows from condition (\ref{eq:cond_S})
that
\begin{align*}
  \int_{\mathcal{Z}_1^o} S(K(Z)) \theta(\d Z)
  = 2 \pi \int_{\mathcal{Z}_1^o} R(Z) \theta(\d Z)
  \le 2\pi \sqrt{\varepsilon - \tfrac{1}{\pi\lambda}},
\end{align*}
i.e.\ the new condition is that the expectation of the radius of a typical
cylinder is less or equal than $\sqrt{\varepsilon - \frac{1}{\pi\lambda}}$.

Furthermore, it follows from \eqref{eq:optimization-problem-new} that maximizing
$p$ is equivalent to maximizing $\int_{\mathcal Z_1^o} R(Z)^2 \, \theta(\d Z)$.

The above calculation shows that the volume fraction of $70\% - 80\%$ in the
optimized gas diffusion layer of a fuel cell can be achieved best by taking
fibers with circular cross sections, relatively small mean radius and high
variance of this radius.

\begin{figure}
  {%
  \centering
  \includegraphics[width=10cm]{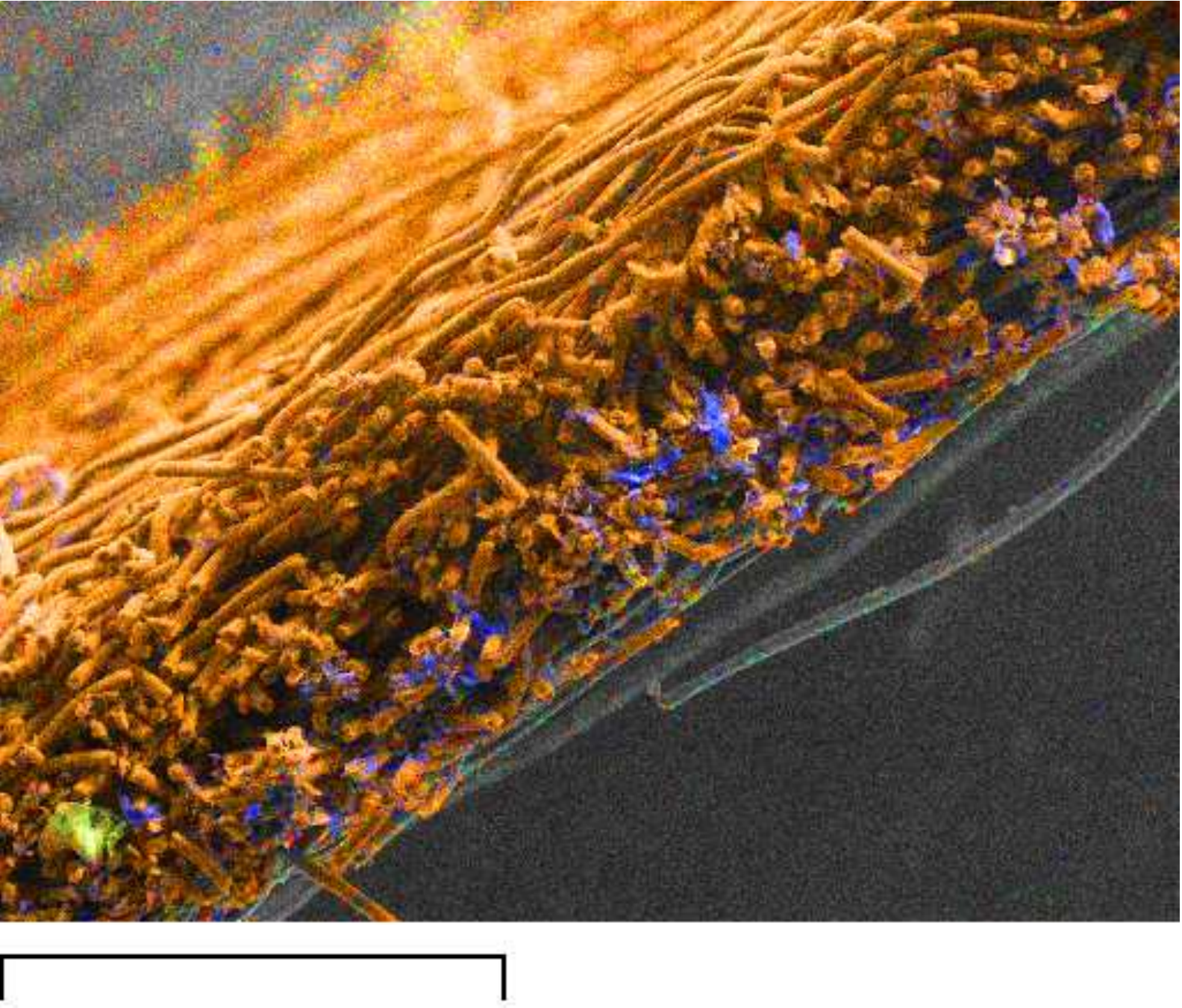}\\
  }
  \vspace{-.7cm}
  \hspace*{3cm} \huge 300 $\mu$m
  \caption{Microscopic picture of the gas diffusion layer}\label{figGDLside}
\end{figure}

Figure~\ref{figGDLside} shows that cross sections of fibers of gas diffusion
layers are almost circular. There are also gas diffusion layers with a little
variance in the fiber radii, although they are mostly nearly constant. Anyhow
the variance of the fiber radii is of course limited, since it is impossible to
produce fibers with an arbitrarily large radius.

We have to remark that from a practical point of view the optimization
problem~(\ref{eq:optimization-problem}) is not well posed. For the construction
of gas diffusion layers, mainly the intensity of the fibers $\lambda$ can be
varied. Hence a practically relevant optimization should involve maximizing the
volume fraction $p$ with respect to $\lambda$ as well. Since the latter problem
is much more involved than the one discussed here, it would go beyond the scope
of this paper.
\\[2ex]
\textbf{Acknowledgements} We would like to thank Werner Nagel, Rolf Schneider, Dietrich Stoyan, Wolfgang
  Weil, and anonymous referees for their useful comments which helped us to
  improve the paper. Furthermore, we are indebted to Christoph Hartnig and
  Werner Lehnert for discussions about fuel cells and for providing
  Figures~\ref{figToray} and~\ref{figGDLside}.

\bibliographystyle{spbasic}      


\begin{thebibliography}{21}
\providecommand{\natexlab}[1]{#1}
\providecommand{\url}[1]{{#1}}
\providecommand{\urlprefix}{URL }
\expandafter\ifx\csname urlstyle\endcsname\relax
  \providecommand{\doi}[1]{DOI~\discretionary{}{}{}#1}\else
  \providecommand{\doi}{DOI~\discretionary{}{}{}\begingroup
  \urlstyle{rm}\Url}\fi
\providecommand{\eprint}[2][]{\url{#2}}

\bibitem[{Ambartzumian(1990)}]{amb90}
Ambartzumian RV (1990) Factorization Calculus and Geometric Probability,
  Encyclopedia of Mathematics and Its Applications, vol~33. Cambridge
  University Press, Cambridge

\bibitem[{Berryman(1987)}]{berryman87}
Berryman JG (1987) Relationship between specific surface area and spatial
  correlation functions for anisotropic porous media. J Math Phys 28:244--245

\bibitem[{Corte and Kallmes(1962)}]{CorKal61}
Corte H, Kallmes O (1962) Formation and structure of paper: Statistical
  geometry of a fiber network. Transactions 2nd Fundamental Research Symposium
  1961, Oxford pp 13--46

\bibitem[{Davy(1978)}]{davy78dis}
Davy P (1978) Stereology --- a statistical viewpoint. PhD thesis, Australian
  National University, Canberra

\bibitem[{Helfen et~al(2003)Helfen, Baumbach, Schladitz, and Ohser}]{hbso03}
Helfen L, Baumbach T, Schladitz K, Ohser J (2003) Determination of structural
  properties of light materials by three--dimensional synchrotron--radiation
  imaging and image analysis. G I T Imaging \& Microscopy 4:55--57

\bibitem[{Hoffmann(2009)}]{hoffm07}
Hoffmann LM (2009) {M}ixed {M}easures of {C}onvex {C}ylinders and {Q}uermass
  {D}ensities of {B}oolean {M}odels. Acta Appl Math 105(2):141--156

\bibitem[{K{\"o}nig and Schmidt(1991)}]{kon:s}
K{\"o}nig D, Schmidt V (1991) Zuf{\"a}llige {Punktprozesse}. Teubner, Stuttgart

\bibitem[{Manke et~al(2007)Manke, Hartnig, Gr{\"u}nerbel, Lehnert, Kardjilov,
  Haibel, Hilger, and Banhart}]{mhglkhhb}
Manke I, Hartnig C, Gr{\"u}nerbel M, Lehnert W, Kardjilov N, Haibel A, Hilger
  A, Banhart J (2007) Investigation of water evolution and transport in fuel
  cells with high resolution synchrotron {X}-ray radiography. Applied Physics
  Letters 90:174,105

\bibitem[{Matheron(1975)}]{ma}
Matheron G (1975) Random Sets and Integral Geometry. J. Wiley \&\ Sons, New
  York

\bibitem[{Mathias et~al(2003)Mathias, Roth, Fleming, and Lehnert}]{maroflle}
Mathias M, Roth J, Fleming J, Lehnert W (2003) Diffusion media materials and
  characterisation. Handbook of Fuel Cells III

\bibitem[{Molchanov et~al(1993)Molchanov, Stoyan, and Fyodorov}]{mo:st93}
Molchanov IS, Stoyan D, Fyodorov KM (1993) Directional analysis of planar fibre
  networks: Application to cardboard microstructure. J Microscopy 172:257--261

\bibitem[{Mukherjee and Wang(2006)}]{MukWan06}
Mukherjee PP, Wang CY (2006) Stochastic microstructure reconstruction and
  direct numerical simulation of the {PEFC} catalyst layer. Journal of the
  Electrochemical Society 153(5):A840--A849

\bibitem[{Ohser and M{\"u}cklich(2000)}]{OhserMueckl00}
Ohser J, M{\"u}cklich F (2000) Statistical Analysis of Microstructures in
  Materials Science. J. Wiley \&\ Sons, Chichester

\bibitem[{Schladitz et~al(2006)Schladitz, Peters, Reinel-Bitzer, Wiegmann, and
  Ohser}]{sprwo06}
Schladitz K, Peters S, Reinel-Bitzer D, Wiegmann A, Ohser J (2006) Design of
  acoustic trim based on geometric modelling and flow simulation for
  non--woven. Computational Materials Science 38(1):56--66

\bibitem[{Schneider(1987)}]{schn87}
Schneider R (1987) Geometric inequalities for {P}oisson processes of convex
  bodies and cylinders. Results in Mathematics 11:165--185

\bibitem[{Schneider(1993)}]{schn}
Schneider R (1993) Convex Bodies. The {Brunn--Minkowski} Theory. Cambridge
  University Press, Cambridge

\bibitem[{Schneider and Weil(2008)}]{SchneiWeil08}
Schneider R, Weil W (2008) Stochastic and {I}ntegral {G}eometry. Probability
  and {I}ts {A}pplications, Springer

\bibitem[{Serra(1982)}]{se82}
Serra J (1982) Image Analysis and Mathematical Morphology. Academic Press,
  London

\bibitem[{Spodarev(2002)}]{Spod02_1}
Spodarev E (2002) {C}auchy--{K}ubota--type integral formulae for the
  generalized cosine transforms. Izv Akad Nauk Armen, Mat [J Contemp Math Anal,
  Armen Acad Sci] 37(1):47--64

\bibitem[{Stoyan et~al(1995)Stoyan, Kendall, and Mecke}]{skm}
Stoyan D, Kendall WS, Mecke J (1995) Stochastic Geometry and its Applications,
  2nd edn. J. Wiley \&\ Sons, Chichester

\bibitem[{Weil(1987)}]{w87}
Weil W (1987) Point processes of cylinders, particles and flats. Acta Appl Math
  9:103--136

\end{thebibliography}
\end{document}